\documentclass[a4paper,10pt]{amsart}
\usepackage[utf8]{inputenc}
\usepackage[english]{babel}
\usepackage{verbatim,color}
\usepackage{amsmath,latexsym,amssymb,verbatim}
\usepackage{graphicx}
\usepackage{enumerate}

\newtheorem{lema}{Lemma}[section]

\newtheorem{prop}[lema]{Proposition}
\newtheorem{coro}[lema]{Corollary}
\theoremstyle{definition}
\newtheorem{defi}[lema]{Definition}
\newtheorem{rema}[lema]{Remark}

\newtheorem{exam}[lema]{Example}
\newtheorem{problem}[lema]{Problem}

\newcommand{\card}{\operatorname{card}}

\newcommand{\Aff}{\operatorname{Aff}}
\newcommand{\Aut}{\operatorname{Aut}}

\newcounter{teoremaganso}

\newtheorem {bigtheo} [teoremaganso] {Theorem}

\def\sideremark#1{\ifvmode\leavevmode\fi\vadjust{\vbox to0pt{\vss 
      \hbox to 0pt{\hskip\hsize\hskip1em           
 \vbox{\hsize2cm\tiny\raggedright\pretolerance10000
 \noindent #1\hfill}\hss}\vbox to8pt{\vfil}\vss}}}%

\newcommand{\C}{\mathbb{C}}

\title[Infinitesimal Center Problem]{
Infinitesimal Center Problem on zero cycles \\ and \\ the composition conjecture 
}
\author{A. \'Alvarez, J.L. Bravo, C. Christopher, P. Marde\v si\'c}

\thanks{The first two authors are supported by Ministerio de Economía y Competitividad through the project MTM2017-83568-P (AEI/ERDF, EU) and also partially supported by Junta de Extremadura/FEDER Grant Number IB18023. The first and second authors are also partially supported by Junta de Extremadura/FEDER Grants Numbers GR18001 and GR18023, respectively.
The last author was supported by Croatian Science Foundation (HRZZ) grant PZS-2019-02-3055 from
Research Cooperability funded by the European Social Fund and by  
EIPHI Graduate School (contract ANR-17-EURE-0002).
}

\begin{document}

\begin{abstract}
We study the analogue of the classical infinitesimal center problem in the plane, but
for zero cycles. We define the displacement function in this context and prove that it is identically zero if and only if the deformation has a composition 
factor. That is,  we prove that here the composition conjecture is true, in contrast with the tangential center problem on zero cycles. Finally, we give 
examples of applications of our results. 
\end{abstract}

\maketitle

\section{Motivation and the main result}

The aim of this paper is to solve the zero dimensional version of the infinitesimal center problem. Zero dimensional means that the problem
is not in terms of families of planar vector fields and one dimensional closed curves (one cycles), but in terms of zero cycles, as we will 
explain in the sequel. As the problem can be stated in purely algebraic terms, we first introduce the problem and then, we detail the 
motivation and some applications.

\subsection{Infinitesimal center and Hilbert 16th problems}

Given a polynomial $f\in\C[z]$, of degree $m\geq 1$, a zero cycle $C$ of $f$ is a multivalued function
\begin{equation*}
C(t)=\sum_{i=1}^m n_i z_i(t), 
\end{equation*}
where the $z_i(t)$ are roots of $f(z_i(t))=t$, and $\sum_{i=1}^m n_i=0$. Consider a small polynomial deformation $f(z)+\epsilon g(z)=t$ of $f = t$ of any degree,
and the deformation induced zero cycle,
\[
C_\epsilon(t)=\sum_{i=1}^m n_i z_i(t,\epsilon), 
\]
where $z_i(t,\epsilon)$, $i=1,\ldots,m$, are the multivalued functions determined by $(f+\epsilon g)(z_i(t,\epsilon))=t$, $z_i(t,0)=z_i(t)$.
We shall always assume that the variable $t$ is restricted to non-critical values of the polynomial $f(z)$ so that the deformation is well-defined.  
\begin{defi} \label{def:displ}
The \emph{displacement function $\Delta$ of the deformation $f(z)+\epsilon g(z)=t$ along the zero cycle $C$} is defined by
\begin{equation*}
\Delta(t,\epsilon)=\int_{C_\epsilon(t)} f,
\end{equation*}
where by definition
\[
\int_{C_\epsilon(t)} f=\sum_{i=1}^m n_i f(z_i(\epsilon,t)).
\]
Note that, since $\sum_{i=1}^m n_i=0$, 
\[
\Delta(t,\epsilon)=\int_{C_\epsilon(t)} f=\int_{C_\epsilon(t)} \left(f+\epsilon g-\epsilon g\right)=-\epsilon\int_{C_\epsilon(t)} g.
\]
\end{defi}
 
We formulate two problems: 

\begin{problem} \label{infcen}\emph{The infinitesimal center problem for zero cycles:}\newline
Caracterize those polynomials $f$, with their cycles $C$ and deformations $g$, such that the displacement function $\Delta$ of the deformation along $C$ is identically 
zero. 
\end{problem}

\begin{problem}\label{inf16} \emph{The infinitesimal Hilbert 16\textsuperscript{th} problem for zero cycles:}\newline
Bound the number of zeros of the displacement function $\Delta$ of the deformation $f+\epsilon g$ along cycles $C$ of $f$ as a function of the degree of 
$f(z)+\epsilon g(z)$ as a polynomial in $z$. 
\end{problem}
 
The two problems are toy examples of the corresponding two problems for small polynomial deformations of Hamiltonian systems in $\C^2$ along one 
cycles, which can be traced back to Poincar\'e and Arnold \cite{A} respectively. We are considering only the polynomial case, but the problems
and some of the results can be extended to a more general setting.

\subsection{Solution of the infinitesimal center problem}

The main result we present here is the solution of the infinitesimal center problem for zero cycles. In order to state our result and to relate it with 
previous results we need to introduce some notations. 

\medskip 

Writing the displacement function in power series of $\epsilon$, we obtain
\begin{equation*}
\Delta(t,\epsilon)=\sum_{i=1}^\infty\epsilon^iM_i(t).
\end{equation*}
We call $M_i(t)$ the $i$-th Melikov function. 
It is easy to check 
that the first Melnikov function is
\[
M_1(t)=-\int_{C(t)}g=\sum_{i=1}^m n_i g(z_i(t)).
\]
That is, a \emph{zero dimensional abelian integral} for the polynomial $g$ along a zero cycle $C(t)=\sum_{i=1}^m n_i z_i(t)$ of $f$, in the terminology of 
Gavrilov and Movasati~\cite{GM}. 

\medskip 

Our problem is analogous to the problem studied by Gavrilov and Movasati in \cite{GM}, but we study the displacement function $\Delta$ whereas they studied 
the first Melnikov function $M_1$. They formulated problems corresponding to Problems \ref{infcen} and \ref{inf16} in terms of the first Melnikov function 
$M_1$ instead of the displacement function $\Delta$. We call their first order problems the \emph{tangential problems} and reserve the adjective 
\emph{infinitesimal} for the study of the displacement function $\Delta$ of a deformation of the identically vanishing function. Example~\ref{exam:1} in 
the final section illustrates the two problems. 
 
\medskip 

Gavrilov and Movasati~\cite{GM} also provided a solution for the tangential Hilbert 16th problem and a solution of a special case of the tangential center 
problem. The complete solution of the tangential center problem is given in \cite{ABM1} (under a generic condition) and in \cite{GP}. See also \cite{ABM2}.
 
\medskip  

A key tool for the solution of the tangential center problem is the composition 
condition. Next, we will define it in our context.

Assume that there exist $h,\tilde f,\tilde g\in\C[z]$ such that $f(z)=\tilde f(h(z))$, $g(z)=\tilde g(h(z)$.
Then, for each cycle $C_\epsilon(t)=\underset{i=1}{\stackrel{m}{\sum}} n_i z_i(t,\epsilon)$ of $f+\epsilon g$, we
define the \emph{projected cycle} $h(C_\epsilon(t))$ of $C_\epsilon$ by $h$ as the cycle of the perturbation $\tilde f+\epsilon \tilde g = t$ defined by
\[
h(C_\epsilon(t)) =  \underset{h(z_i(t,\epsilon) )}{\sum} \left( \underset{h(z_j(t,\epsilon) )=h(z_i(t,\epsilon) )}{\sum} n_j \right) h(z_i(t,\epsilon) ).
\]
We say that this projected cycle is \emph{trivial} if 
\[ 
\underset{h(z_j(t,\epsilon))=h(z_i(t,\epsilon))}{\sum} n_i =0,\quad i=1,\ldots,m.
\]

\begin{defi}
We say that $f,g\in\C[z]$ and a cycle $C$ of $f$ satisfy the {\em composition condition} if there exist polynomials
$\tilde f,\tilde g,h \in \mathbb{C}[z]$ such that 
\end{defi}
\begin{equation*}
f(z)=\tilde f(h(z)),\quad g(z )=\tilde g(h(z )),
\end{equation*}
\\and for every $\epsilon$, the perturbed cycle $C_\epsilon$ projected by $h$ is trivial. That is,
\begin{equation}\label{eq:cc2}
\sum_{h(z_i(t,\epsilon) )=h(z_j(t,\epsilon) )} n_i = 0, \quad i=1,\ldots,m.
\end{equation}

\medskip 

Now, we can state our main result.

\begin{bigtheo}\label{theo:main}
A deformation, $f+\epsilon g = t$, has an infinitesimal center for a cycle $C$ of $f$, i.e., $\Delta(t,\epsilon)\equiv 0$ for all $t$ and all $\epsilon$ small enough, if and only if $f$, $g$, $C$ satisfy the composition condition. 
\end{bigtheo}

The sufficiency of the composition condition is easy. Indeed, 
by \eqref{eq:cc2} we get that
\[
\Delta(t,\epsilon) = \sum_{j=1}^m n_j f(z_j(t,\epsilon)) = \sum_{j=1}^m n_j \tilde f(h(z_j(t,\epsilon)))
=\int_{h\left(C_\epsilon\right)} \tilde f\equiv 0.
\]
The converse is more difficult and will be proved in the next section. 

\begin{rema}
We note that the composition conjecture is {\em not true} for the tangential center problem, as showed by Pakovich's example~\cite{P}. 
(See Examples~\ref{exam:1} and \ref{exam:2}.) It is true, however, for the tangential center problem on simple cycles (Theorem 1.7~\cite{CM}.) 
\end{rema}

\begin{rema}
In Theorem~\ref{theo:main} we assume that $f,g$ are polynomials, but the problem can also be considered for analytic functions $f,g$. In particular, the 
function $f$ could have an infinite number of isolated fibers $z_i(t)$, so the cycle sould be considered with finite support. 
\end{rema}


\subsection{Motivation}

The motivation for this paper comes from the study of polynomial vector fields in the plane. Orbits of polynomial vector fields are relatively simple and 
most interesting are periodic orbits, which can belong to continuous families of periodic orbits, which we call by abuse a \emph{center}, or can be 
isolated, called \emph{limit cycles}. 

Two widely open classical problems, the \emph{center problem} (Poincar\'e) and the \emph{Hilbert 16-th problem}, are related to these two situations. 

The center problem asks for a geometric caracterization of polynomial vector fields in the plane having a center. The 16-th Hilbert problem asks for a 
bound, as a function of the degree, for the number of limit cycles. 

Each of these problems has its \emph{infinitesimal version}. One starts with a polynomial system $dF=0$ having a 
family of cycles $C$ called a center.
Consider its deformation 
\begin{equation}\label{deformation}
dF+\epsilon \omega=0
\end{equation}
of degree $n$, where $\omega$ is a polynomial one-form.

\smallskip
\emph{Infinitesimal center problem on $1$-cycles:} 
Find all polynomial deformations \eqref{deformation} of $dF$ preserving the center i.e. for which the family of periodic solutions is tranformed to a 
nearby family of periodic solutions. 
\smallskip 
 
\emph{Infinitesimal 16-th Hilbert problem} (Arnold~\cite{A}): Give an upper bound for the number of limit cycles born in \eqref{deformation} as a function of the degree $n$.
\smallskip 
 
In order to deal with these problems, one considers a transversal $T$ parametrized by the values of $F$, the (not necessarily closed) trajectory  $C_\epsilon(z)$ of \eqref{deformation} with end points on $T$ deforming $C(z)$ and one defines the \emph{displacement function $\Delta$  of the deformation
\eqref{deformation} by 
\begin{equation}\label{eq:displacement}
\Delta(z,\epsilon)=\int_{C_\epsilon}dF=-\epsilon\int_{C_\epsilon}\omega.
\end{equation}}

For the center problem one searches for a characterization when $\Delta$ is identically equal to zero. For the infinitesimal 16-th 
Hilbert problem, one searches for a bound for the number of isolated zeros of $\Delta$. Our Definition \ref{def:displ} of the displacement function of  a deformation along zero cycles is directly inspired by \eqref{eq:displacement}.

The previous problems also have their tangential versions, obtained by expanding in series of 
functions the function $\Delta$ with respect to powers of $\epsilon$ and asking the analogous questions for the first term $M_1(t)$ of $\Delta(t,\epsilon)$. 

A special case is when the planar system can be reduced to a family of Abel equations, 
\[
x'=A(t)x^2+\epsilon B(t)x^3,
\]
where $A$, $B$ are trigonometric polynomials. The tangential center problem for special types of cycles first appeared when studying this 
family~(see \cite{BFY99,BFY00,BFY00-2}) where, for simplicity, Briskin, Françoise and Yomdin considered $A,B$ polynomials rather than trigonometric polynomials. In that context, the 
\emph{composition conjecture} was formulated~\cite{BFY00}, conjecturing that a certain sufficent condition (the composition 
condition) for vanishing of the abelian integrals was also necessary. The composition condition defines all the irreducible components of the center 
variety in many families of Abel equations~\cite{CGM2}, in some planar systems (see \cite{ZR,ZY}), and accounts for most of the irreducible components when studying the 
tangential centers of the Abel equation at infinity~\cite{BRY}.  However, not all centers satisfy the compostion conjecture in the trigonometric Abel equation~\cite{Al}, 
or even the polynomial Abel equation~\cite{GG}.

The tangential center problem on zero cycles also has appeared when studying hyperelliptic planar systems~\cite{CM}.
In the last section, we show that the infinitesimal center problem also appers in these contexts, although
 in a more general version than the problem solved in this paper. 

\section{Monodromy group of perturbations and proof of the main result}

In this section we prove Theorem~\ref{theo:main}. The key of the proof will be to define 
a convenient monodromy for the deformation. To that end, we start recalling the monodromy
of a polynomial and then extending it to the deformation of a polynomial. 

Given a nonconstant polynomial $f \in \C[z]$, recall that $z_0\in\C$ is a \emph{critical point} 
of $f$ if $f'(z_0)=0$, and its associated \emph{critical value} is $t=f(z_0)$. If $t$ is not a 
critical value, we say that $t$ is \emph{regular}.

We will denote by $\Sigma'$ the set of all critical values of $f$, which is a finite set. Let $m>1$ be the degree of $f$. Then, for $t \in \C \setminus 
\Sigma'$ the set $f^{-1}(t)$ consists of $m$ different points $z_i(t)$, $i=1,\ldots,m$. By the implicit function theorem, one can push locally 
each solution $z_i(t)$ to nearby values of $t$ thus defining multi-valued analytic functions $z_i(t)$, $t \in \C \setminus \Sigma'$.

For a polynomial $f$ of degree $m$ with $\Sigma'$ the set of critical values, each loop based at $t_0 \in \mathbb{C} \backslash \Sigma'$ defines a 
permutation of the roots $z_1(t_0),\ldots, z_m(t_0)$ of $f(z)=t_0$. Thus, we have a mapping from the fundamental group $\pi_1(\mathbb{C} \backslash 
\Sigma', t_0)$ to the automorphism group $Aut(f^{-1}(t_0))$, whose image forms a group $G_f$,  called the \emph{monodromy group} of $f$. 
 
The monodromy group $G_f$ acts transitively on the fibre $f^{-1}(t_0)$ (see~\cite{CM}).
Moreover, $G_f \subseteq S_m$ is the Galois group of the Galois extension of $\mathbb{C}(t)$ by the $m$ pre-images $z_1(t),\ldots, z_m(t)$ of $t 
\in \mathbb{C} \backslash \Sigma'$ by $f$ (see, for instance, Theorem~8.12 of \cite{F}). That is,
\[
G_f = \Aut_{\mathbb{C}(t)}\mathbb{C}(z_1,\ldots, z_m).
\]
The monodromy group $G_f$ induces an action on cycles. In fact, reordering the preimages $z_i(t)$ appearing in a cycle $C$ after the action of 
the monodromy, we can consider the monodromy as acting on the coefficients $n_i$ of a zero cycle of $f$ by permuting them:
\[
\sigma(C(t)) = \sum_{i=1}^m n_i \sigma(z_i(t)) = \sum_{i=1}^m n_i z_{\sigma(i)}(t) = \sum_{i=1}^m n_{\sigma^{-1}(i)} z_i(t).
\]

\subsection{Monodromy group of a polynomial perturbation}

Now we define the monodromy group $G$ for a family of polynomials $f(z)+\epsilon g(z)$, considered as polynomials in $z$. Let us denote
\[
\begin{array}{rccl} 
F: &  \mathbb{CP}\times\C & \to & \mathbb{CP} \\ & (z, \epsilon) & \mapsto & F(z,\epsilon) =f(z)+\epsilon g(z), z\in \mathbb{C},\text{ and }F(\infty,\epsilon)=\infty
\end{array}
\]
and 
\[
\begin{array}{rccl} 
H: &  \mathbb{CP}\times\C & \to & \mathbb{CP} \times \C \\ & (z, \epsilon) & \mapsto & H(z,\epsilon) =(F(z,\epsilon),\epsilon).
\end{array}
\]

Let $D(t,\epsilon)$ be the discriminant of $f(z)+\epsilon g(z)-t$ as a polynomial in $z$, and $c(\epsilon)$ the coefficient of the monomial of the 
highest degree. Define
\[
\Sigma := \{ (t,\epsilon) \in \C^2: c(\epsilon)D(t,\epsilon)=0\}\cup \{\infty\}\times \C.
\]

The fiber of every $(t,\epsilon) \in \mathbb{CP} \times \C \setminus \Sigma$, consists of $n$ different points $(z_i(t,\epsilon),
\epsilon)$, where $n$ is the degree of $F$ as a polynomial in $z$. By Ehresmann's fibration lemma~\cite{E}, the map 
 $H: \mathbb{CP}\times\C \setminus 
H^{-1}(\Sigma) \to \C^2 \setminus \Sigma$ defines a fibration whose fibers are zero dimensional, given by the union of $n$ distinct points $z_i(\epsilon,t)$. 

The fact that $H$ defines a fibration guarantees that any closed path $\gamma$ in the basis $\mathbb{CP}\times\C \setminus \Sigma$ at a point 
$(t_0,\epsilon_0)$ lifts to a path  joining two points, $z_i(t,\epsilon)$ and $z_j(t,\epsilon)$, of the fiber. Moreover, closed paths homotopic to 
$\gamma$ lift to homotopic paths between the points of the fiber $z_i(t,\epsilon)$ and $z_j(t,\epsilon)$. As $\mathbb{CP}\times\C 
\setminus \Sigma$ is path-connected, its fundamental groups with different base poins are conjugated and one can consider the basepoint free homotopy 
group. We obtain a well-defined application from $\pi_1(\mathbb{CP}\times\C \setminus \Sigma, (t,\epsilon))$ to $Sym
(H^{-1}(t,\epsilon)) \simeq S_{n}$. We call the subgroup $G$ of $Sym(H^{-1}(t,\epsilon))$ thus obtained as the image of $\pi_1
(\mathbb{CP}\times\C \setminus \Sigma, (t,\epsilon))$ the \emph{monodromy group} $G$ of $f+\epsilon g$.

\medskip 

Consider also the projection $p_2\colon \mathbb{CP}\times\C \to \C$, $p_2(t,\epsilon)=\epsilon$. Define
\[
\Sigma_0=\{(t,\epsilon)\in\Sigma\colon dp_{2}|_{\Sigma=0}.\}
\]
Note that since $\Sigma$ is an algebraic variety, the projection of the points where the variety is orthogonal to the projection is just a finite set. 

\begin{figure*}[ht]
\begin{center}
\begin{picture}(150,220)
\put(0,0){\includegraphics[width=150pt]{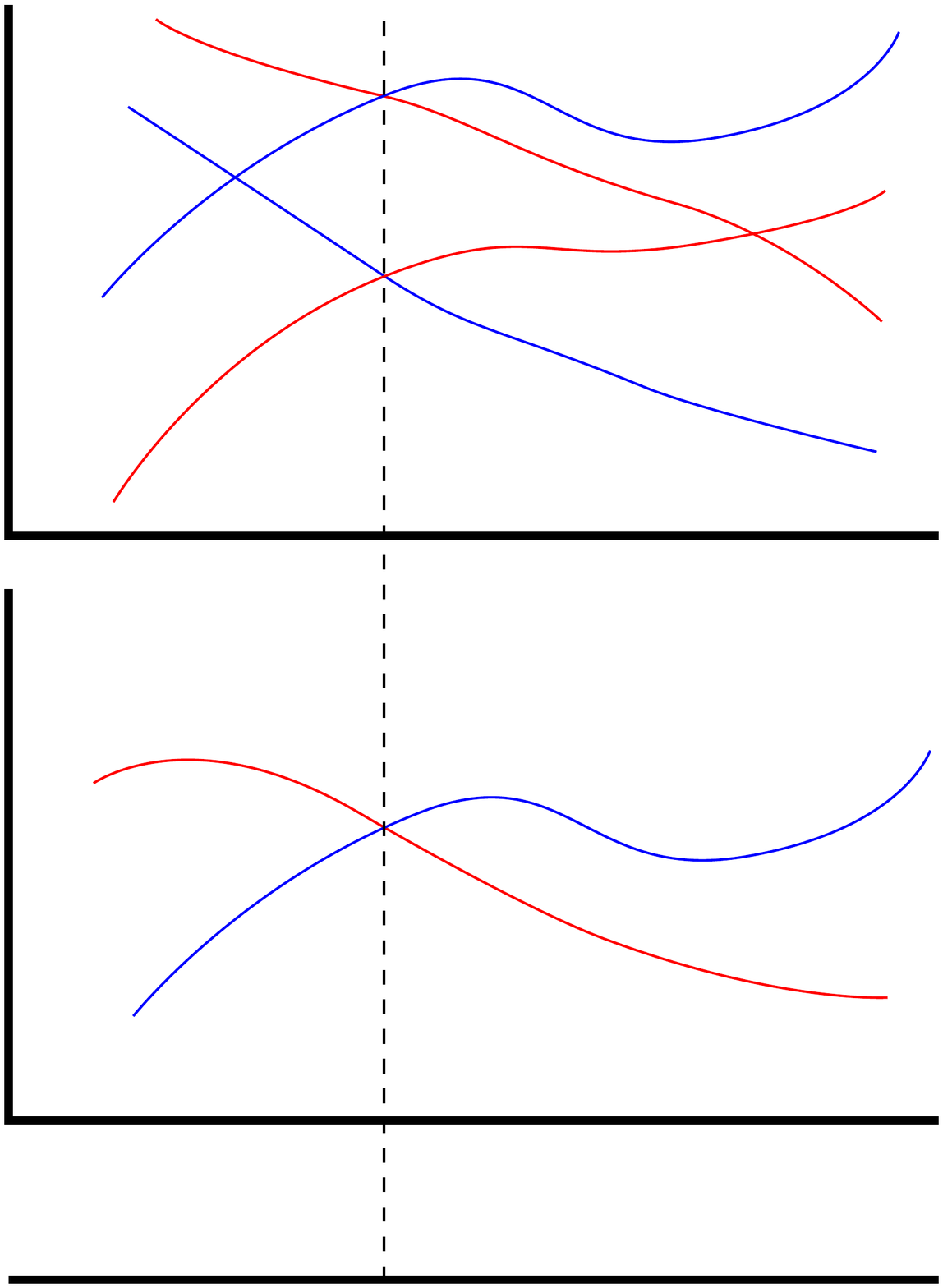}}
\put(145,17){$\epsilon$}
\put(145,40){$\epsilon$}
\put(145,118){$\epsilon$}
\put(10,190){$z$}
\put(10,110){$t$}
\put(100,190){$H^{-1}(\Sigma)$}
\put(120,85){$\Sigma$}
\put(60,72){$\Sigma_0$}
\put(60,16){$p_2(\Sigma_0)$}
\put(155,160){$\mathbb{CP}\times \C$}
\put(170,125){$\Bigg\downarrow$}
\put(180,125){$H$}
\put(155,85){$\mathbb{CP}\times \C$}
\put(170,55){$\Bigg\downarrow$}
\put(180,55){$p_2$}
\put(170,20){$\C$}
\end{picture}
\end{center}
\caption{Diagram of the fibrations}
\end{figure*}

\begin{prop} 
Given a closed path $\gamma$ in the basis $\mathbb{CP}\times\C \setminus \Sigma$ and $\epsilon_0\in\C\backslash 
p_2(\Sigma_0)$, the path $\gamma$ is homotopic within $\mathbb{CP}\times\C \setminus \Sigma$  to a closed path in 
$(\C\setminus \Sigma_{\epsilon_0})\times\{\epsilon_0\}$, where $\Sigma_{\epsilon_0}$ is the set of critical values of $f(z)+\epsilon_0\, g(z)$.
\end{prop}

\begin{proof}
The proof relies on Thom's first isotopy lemma (see e.g.~\cite[p.180]{Ads}), which we apply to 
the mapping
\[
p_2 \colon W\to \C\backslash p_2(\Sigma_0),
\]
where $W=(\mathbb{CP}\times \C)\backslash \Sigma_0$, is stratified as
\[
\Sigma\backslash\Sigma_0<W\backslash\Sigma.
\]
The map $H$ restricted to $W$ is proper and its restriction to $\Sigma$ is a submersion. Therefore, 
\[
p_2 \colon W\to \C\backslash p_2(\Sigma_0)
\]
is a topologicaly locally trivial fibration compatible with the stratification.
We consider a connection associated with this stratification. 

Now, consider a closed path in $W\backslash\Sigma$. Then, we can deform the path by the conection without crossing $\Sigma$ until we obtain a closed path 
in $\mathbb{CP}\times\{\epsilon_0\}$, with $\epsilon_0\not\in \Sigma_0$. 
\end{proof}

Note that the monodromy induced by the original closed path and the deformed path are conjugated. 

For fixed $\epsilon\in\mathbb{C}$, denote by $G_\epsilon$ the monodromy group of $f(z)+\epsilon g(z)$. 
The next result shows the relation between $G$ and $G_\epsilon$.

\begin{coro}\label{coro:G}
For every $\epsilon\in\mathbb{C}$ such that $c(\epsilon)\neq 0$, $G_\epsilon$ is a subgroup of $G$, up to conjugation, and for every $\epsilon \not \in 
p_2(\Sigma_0)$, $G_\epsilon$ is $G$, up to conjugation. Moreover, $G$ is the Galois group of the 
Galois extension of $\mathbb{C}(t,\epsilon)$ by the $n$ preimages $z_1(t,\epsilon),\ldots,z_n(t,\epsilon)$.
\end{coro}

\begin{proof}
Let us consider a base point $(t_0,\epsilon_0) \in \C^2 \setminus \Sigma$ and let $G$ be the monodromy group induced by closed paths with that base point. 
The monodromy group $G_{\epsilon_0}$ is generated by the closed paths based at $t_0$, which are in the hyperplane $(\C \setminus \Sigma_{\epsilon_0}) 
\times \{ \epsilon_0\} \subset \C^2 \setminus \Sigma$. Hence, they generate the same element in $G$ and obviously $G_{\epsilon_0} \subset G$. 

On the other hand, by the previous proposition, closed paths in $\mathbb{CP}\times\C \setminus \Sigma$ based at $(t_0,\epsilon_0)$ with 
$\epsilon_0\not\in p_2(\Sigma_0)$ can be continuously deformed into closed paths in the hyperplane $\{\epsilon = \epsilon_0\}$. Therefore, permutations 
corresponding to closed paths in $G$ also belong to $G_{\epsilon_0}$, so $G\subset G_{\epsilon_0}$.

Let $\bar G$ denote the Galois group of the 
Galois extension of $\mathbb{C}(t,\epsilon)$ by the $n$ preimages $z_1(t,\epsilon),\ldots,z_n(t,\epsilon)$, that is
\[
\bar G = \Aut_{\C(t,\epsilon)} (\C(z_1(t,\epsilon),\ldots,z_n(t,\epsilon))).
\]
Since the elements of $G$ are automorphisms of the fibers, then $G$ is a subgroup of $\bar G$. 
On the other hand, we have the projection $\bar G\to G_\epsilon$, as any automorphism 
can be restricted to a hyperplane. Take $\epsilon,\epsilon'\not\in p_2(\Sigma_0)$, then 
the following commutative diagram holds
\[
\begin{matrix}
G_\epsilon & & \longleftrightarrow & & G_{\epsilon'}\\
&\nwarrow &   & \nearrow &\\
&  & \bar G &  & 
\end{matrix}
\]
If the image of an element of $\bar G$ is the identity element of $G_\epsilon$, 
then the diagram above proves that it is the identity element of $G_{\epsilon'}$ 
for every $\epsilon'\not\in p_2(\Sigma_0)$ and, in consequence, the element of $\bar G$ is
the identity element. Therefore, the morphism $\bar G\to G_\epsilon =G$ is
injective, and $\bar G$ and $G$ are conjugated. 

\end{proof}

\subsection{Proof of the main result}
In order to prove Theorem~\ref{theo:main}, we distinguish different cases, given by the following proposition.
We will follow \cite{CM} (see also \cite{CL} or \cite{ABM1} for more details).

A permutation group $G$ acting on $X$ is said to be {\em imprimitive} if there exists a proper subset $B\subset X$, $\card(B)>1$, such that for any 
$\sigma \in G$, $\sigma(B)\cap B=\varnothing$ or $\sigma(B)=B$. If $G$ is not imprimitive, then, it is called {\it primitive}.

\begin{prop}\label{prop:BS}
Let $f,g\in\mathbb{C}[z]$ and let $G$ be the monodromy group of $f+\epsilon g$. Then, one of the following 
cases holds:
\begin{enumerate}[(i)] 
 \item\label{c1} $G$ is two transitive.
 \item $G$ is isomorphic to the monodromy group of $z^p$ with $p$ prime.
 \item $G$ is isomorphic to the monodromy group of $T_p(z)$ with $p$ prime.
 \item\label{c4} $G$ is imprimitive.
\end{enumerate}
\end{prop}

\begin{proof}
Let $c_k(\epsilon)$, $k=0,\ldots,n:=\max(\deg f,\deg g)$, be the coefficients of the polynomial $f(z)+\epsilon g(z)$ as a polynomial in $z$. That is,
\[
f(z)+\epsilon g(z)=c_n(\epsilon) z^n+c_{n-1}(\epsilon)z^{n-1}+\ldots+c_0(\epsilon),
\]
where $c_k$ are affine in $\epsilon$. 

Choose $\epsilon_0$ such that $c_n(\epsilon_0)\neq 0$. If $t$ is large, then $z_k(t,\epsilon_0)$ can be expanded as 
\[
z_k(t,\epsilon_0)=\omega^k c(\epsilon_0)^{1/n} t^{1/n} + O(t^{1/n-1}),
\]
where $\omega$ is a primitive nth-root of unity. Hence, if we consider a cycle with $t$
large enough, it produces a cyclic permutation of the $z_k$. Therefore,
there exists a cycle $(1,\ldots,n)\in G$.

Now, applying Burnside-Schur theorem~(see \cite{DM}), either $G$ is imprimitive or, two transitive, or permutationally isomorphic to the affine 
group $\Aff(p)$, where $p$ is a prime. If $G$ is two-transitive, then we are in case \eqref{c1}. Assume $G$ is permutationally isomorphic to the 
affine group $\Aff(p)$, where $p$ is a prime.

Arguing as in \cite{CM}, we obtain that for any fixed $\epsilon$, the monodromy group of $f+\epsilon g$ is either isomorphic to the monodromy group 
of $z^p$ or to the monodromy group of $T_p$, for $p$ prime. By Corollary~\ref{coro:G}, the same holds for $G$ for any $\epsilon\not\in p_2(\Sigma_0)$.

\end{proof}

The case $G$ two-transitive is generic. In particular, it contains the case when all critical points are of Morse type with different critical values. 
We first solve the infinitesimal center problem in this case. 

\begin{prop}\label{2trans}
Assume that for some polynomials $f,g$, the group  $G$ is two-transitive. Let $C$ be a zero cycle of $f$. Then, $f+\epsilon g$ has an infinitesimal 
center for the cycle $C$ if and only if $f, g, C$ satisfy the composition condition.
\end{prop}

\begin{proof} 
Fix $\epsilon_0\not\in \Sigma_0$. For any fiber $z_i(t,\epsilon_0)$, denote $G_i$ the stabilizer of the fiber. Acting on the zero cycle by $G_i$, 
since $G$ is two-transitive, we get
\[
0=\sum_{\sigma\in G_i} \sum_{j} n_{\sigma(j)} f(z_j)= |G_i| n_i f(z_i) +|G_i|\frac{\sum_{j\neq i} n_j}{n-1}\sum_{j\neq i} f(z_j),
\]
where $n$ is the degree of $f+\epsilon g$ as a polynomial in $z$.
Since $C$ is a cycle then $n_i+\sum_{j\neq i} n_j=0$. Assume that $n_i\neq 0$, dividing by $n_i|G_i|$ in the previous equation, we obtain
\[
0=f(z_i)-\frac{1}{n-1}\sum_{j\neq i} f(z_j).
\]
Assume that $n_1,n_2\neq 0$ (reordering the roots if necessary). Then,
\[
\begin{split}
0&=\left(f(z_1)-\frac{1}{n-1}\sum_{j\neq 1} f(z_j)\right)-\left(f(z_2)-\frac{1}{n-1}\sum_{j\neq 2} f(z_j)\right)
\\&=\frac{n}{n-1}f(z_1)-\frac{n}{n-1}f(z_2).
\end{split}
\]
Arguing as in Proposition~\ref{prop:BS}, by Lüroth's theorem, $L=\C(h(z_1))$, so we obtain $f+\epsilon g=F(h(z),\epsilon)$ for a certain 
$F\in \mathbb{C}[z,\epsilon]$, affine in $\epsilon$, and the composition condition holds.
\end{proof}

In the proof of Propostion~\ref{2trans}, we have not used essentially that $f$ and $g$ are polynomials, so the proposition could be generalized 
to the class of rational functions. 

\medskip

\begin{proof}[Proof of Theorem~\ref{theo:main}]
As we mention is the introduction, the composition condition is 
sufficient for $f+\epsilon g$ to have an infinitesimal center for the cycle $C$. 

Now, assume that the deformation $f+\epsilon g$ has an infinitesimal center for the cycle $C$. 
By Proposition~\ref{prop:BS}, the monodromy group $ G$ of $f+\epsilon g$ 
is either two-transitive, equivalent to 
a monomial $z^p$ or a Chebyshev polynomial $T_p$ with $p$ prime,
or  is imprimitive. If $G$ is two-transitive, we conclude by Proposition~\ref{2trans}. We now deal with each of the remaining cases. 

Assume that $G$ is isomorphic to the monodromy group of $z^p$ or $T_p$, with $p$ prime. In the first case, 
$f+\epsilon g$ has a unique critical value for every $\epsilon$ (as $G_\epsilon$ is a subgroup of $G$), and
therefore, $f(z)+\epsilon g(z)=K(\epsilon)(z-a(\epsilon))^p$ for some functions $K,a$. But equaling 
the coefficients, $a$ must be constant, and $K(\epsilon)$ affine, so $f(z)+\epsilon g(z)=K(\epsilon) (z-a)^p$. In the second case,
arguing analogously, the two critical values must remain constant and therefore, $f$ and $g$ 
are multiple of the same Chebyshev polynomial.  
Note that in all cases, the cycle projected by the composition factor is trivial. To resume, we have obtained that if  the group $G$ is primitive, 
then the cycle can by projected to a trivial one by $f$. 

Assume that $G$ is imprimitive. 
Then there exists a subset of fibers $B$ such that for every $\sigma\in G$, we have either $\sigma(B)\cap B=\varnothing$ or $\sigma(B)=B$. Assume $B$ contains the fiber $z_1$. Denote $G_1$ and $G_B$ the stabilizers of $z_1$ and $B$, respectively. We have 
the groups inclusions
\[
G_1\subset G_B\subset G.
\]
Since  $G$ is the Galois group of $f+\epsilon g$, by the fundamental theorem of Galois theory, we have the inclussions of fields
\[
\C(\epsilon)(t)\subset L:=\C(\epsilon)(z_1,\ldots,z_n)^{G_B}\subset \C(\epsilon)(z_1).
\]
Now, applying Lüroth's theorem~(see e.g. \cite{W}), we obtain $f+\epsilon g=F \circ h$ for certain $F, h \in \mathbb{C}(z,\epsilon)$, where $F,h$ have degree $>1$ in $z$. As $F\circ h$ is a polynomial, then for every $\epsilon$ the preimage of $\infty$ by $h$ is a point, $\alpha(\epsilon)$, and the preimage of $\alpha(\epsilon)$ by $F$ is again $\infty$. Then, composing with a convenient Möebius function (see, for instance, Lemma~3.5 of \cite{CM}), we can assume that $\alpha(\epsilon)\equiv \infty$, so $F,h \in \mathbb{C}[z,\epsilon]$. Since the composition $F(h(z,\epsilon),\epsilon)$ has degree one in $\epsilon$, then either $F$ or $h$ does not deppend on $\epsilon$. Moreover, as $\deg F>1$, $F$ must be affine in $\epsilon$, that is, $F(z,\epsilon)=F_0(z)+\epsilon F_1(z)$, for certain $F_0,F_1\in\mathbb{C}[z]$. 

That is,  $f(z)+\epsilon g(z)=F(h(z),\epsilon)$, were $F,h$ are polynomials  and $F$ is affine in $\epsilon$, 
and consider the projected cycle $h(C_\epsilon)$. We consider the projected problem, that is,
\[
0\equiv \int_{C_\epsilon(t)} f=\int_{h(C_\epsilon(t))} F(\cdot,0).
\]
where the degree of $F(\cdot,0)$ is strictly lower than the degree of $f$. 

\medskip 

If the projection $h(C_\epsilon)$ is trivial we conclude. 

\medskip 

If not, we continue projecting until the monodromy group of the projected problem is primitive or until we get a trivial projection. In the first case, 
as proved above, the cycle can be proyected to a trivial one, so we assume we are in the second case. Assume that $h,h_1,h_2,\ldots,h_k$ are the succesive 
projections, that $F,F_1,F_2,\ldots,F_k$ are the succesive composition factors, and that $h_k(h_{k-1}(\ldots h_1(h(C_\epsilon))))\equiv 0$. 
Then, 
\[
f(z)+\epsilon g(z)=F_k((h_k\circ h_{k-1}\circ \ldots \circ h_1\circ h)(z),\epsilon),
\] 
and the cycle projected by $h_k\circ h_{k-1}\circ \ldots \circ h_1\circ h$ is trivial, so it satisfies the composition condition.

\medskip

\end{proof}

\begin{rema}
Recall that the composition condition is not necessary for $g$ to be a solution 
of the {\em tangential} center problem. This is due to the fact that monodromy in the tangential 
center problem is just the monodromy group of $f$, and the problem is 
linear in $g$, so it is only required that the summands in $g$ satisfy the 
composition condition, while in the infinitesimal
center problem, we are considering a group that contains both the monodromy
groups of $f$ and $g$, and therefore, composition factors considered must be common
composition factors of $f$ and $g$. (See Example \ref{exam:1}.)
\end{rema}

\section{Examples and applications}

In this section, we show some examples and applications to illustrate 
the problem and its relation with dynamical systems. The first two
examples show that the problem is not trivial, even in simple cases.
The third example introduces the infinitesimal problem 
for planar vector fields and shows how, for certain planar fields,
the one dimensional problem reduces to a zero dimensional problem.
The last two examples shows the same but for the Abel equation. 
Moreover, in the last example we shall apply the results of 
this paper to obtain a new proof of a recent result. 

\begin{exam}\label{exam:1}{\bf Tangential problem.}
Consider $f(z)=z^6$, and a cycle 
\[
C(t)=\sum_{j=0}^5 n_j z_j(t),\quad z_j(t)=t^\frac{1}{6} e^{j\frac{2\pi i}{6}},
\] 
where $n_0=2$, $n_1=1$, $n_2=-1$, $n_3=-2$, $n_4=-1$, $n_5=1$.

There are two possible decompositions of $f$, with factors $h_1(z)=z^3$ and $h_2(z)=z^2$. Consider a perturbation with $g(x)=z^3+z^2=h_1(z)+h_2(z)$.

The projection of $C(t)$ by $h_1$ consists of identifying the roots $z_j$ such that $h_1(z_j)$ has the same value 
and assign the sum of the weights. That is,
\[
h_1(C(t))=(n_0+n_2+n_4) t^{1/2} + (n_1+n_3+n_5) (-t^{1/2})=0.
\]
Analogously, $h_2(C(t))$ is a trivial cycle.

Recall that 
\[
\Delta_\epsilon(t)=-\epsilon \int_{C(t)} g+O(\epsilon^2).
\]
It is known (see e.g.~\cite{ABM1}) that for this election of $f,g$ and $C$, the return map 
at first order is identically zero, as 
\[
\int_{C(t)} g=\int_{C(t)} h_1+h_2=\int_{C(t)} h_1 + \int_{C(t)} h_2=\int_{h_1(C(t))} z + \int_{h_2(C(t))} z=0.
\]

On the other hand, when calculating 
\[\Delta(t)=-\epsilon\int_{C_\epsilon}g=-\epsilon\int_{C_{\epsilon}}h_1+h_2,\] 
but $C_\epsilon$ is a cycle of $f+\epsilon g$, which has no non-trivial composition factors, so the above argument does not work.

Indeed, as we will prove, the second Melnikov function is not zero, so 
the solution for the tangential problem is not a solution of the infinitesimal problem.
Indeed, if we derive in $f(z_j(t,\epsilon))+\epsilon g(z_j(t,\epsilon))=t$ with respect to $\epsilon$, we obtain
\[
\frac{\partial z_j}{\partial \epsilon}(t,\epsilon)=\frac{-g(z_j(t,\epsilon)) g'(z_j(t,\epsilon))}{f'(z_j(t,\epsilon))}.
\]
Then, $\Delta_\epsilon(t)=\epsilon^2 M_2(t) /2+O(\epsilon^3)$, where $M_2$ is obtained differentiating $\int_{C_\epsilon(t)} g(z)$
with respect to $\epsilon$, obtaining
\[
M_2(t)=\int_{C(t)} \frac{-g(z) g'(z)}{f'(z)}.
\]
But, replacing the roots by its value,
\[
M_2(t)=-\sum_{i=0}^6 n_j \frac{g(z_j(t)) g'(z_j(t))}{f'(z_j(t))}=-\frac{17}{12 t}.
\]
Therefore, it is not an infinitesimal center.

\end{exam}

\begin{exam}\label{exam:2}{\bf A simple monomial perturbation.}
Consider $f(z)=z^4$, a perturbation of the form $g(x)=a z^4+b z^2 + c$, and a generic cycle 
\[C(t)=\sum_{j=0}^3 n_j z_j(t),\quad z_j(1)=i^j.
\] 

By direct computation, it is easy to obtain that the roots of $f(z)+\epsilon g(z)=t$ are
\[
z_j(t,\epsilon)=(-1)^{\lfloor j/2\rfloor}\sqrt{\frac{-\epsilon b+(-1)^j \sqrt{\epsilon^2b^2-4(\epsilon c-t)(1+\epsilon a)} }{2(1+\epsilon a)}}.
\]
Then, 
\[
\Delta_\epsilon(t)=\sum_{j=0}^3 n_j f(z_j(t,\epsilon))=\frac{\epsilon b (-n_0+n_1-n_2+n_3)\sqrt{\epsilon^2b^2-4(\epsilon c-t)(1+\epsilon a)} }{2(1+
\epsilon a)^2}.
\]
That is, if $b=0$ or $n_1+n_3=n_0+n_2$, then $\Delta_\epsilon(t)\equiv 0$ for every $t,\epsilon\in\mathbb{C}$. In the first case, $g(z)=a f(z)+c$, 
and in the second case, if we denote $h(z)=z^2$, then $f(z)=h^2(z)$, $g(z)=a h^2(z)+b h(z)+c$, and, since $n_0+n_1+n_2+n_3=0$, then $n_2+n_0=0$ and 
$n_1+n_3=0$. In particular, the projection of $C_\epsilon$ by $h$ is identically zero. 
\end{exam}

\begin{exam}{\bf Infinitesimal center problem and hyper-elliptic planar systems.}
Let us consider a perturbed center problem in the plane, that is, consider
$F\in\mathbb{C}[x,y]$, a polynomial one-form $\omega$, and the deformation of
the foliation $F(x,y)=t$, 
\[
dF+\epsilon \omega=0.
\]
Take a regular value $t_0$ of $F$ and $\gamma(t_0)\subset F^{-1}(t_0)$ a closed path. Let $T$ be 
a transversal section to the leaves of $F$ parametrized by $t$. Consider the one cycle $\gamma_\epsilon(t)$
obtained by deformation 
of $\gamma(t_0)$ with respect to $t$ and $\epsilon$. 
Note that $\gamma_\epsilon(t)$ is not necessarily closed, as we are considering it 
up to the first intersection with $T$.
Let $\Delta_\epsilon$ be the associated displacement map 
\[\Delta_\epsilon(t)=\int_{\gamma_\epsilon(t)} dF=-\epsilon \int_{\gamma_\epsilon(t)} \omega.\] 
Then, the deformation preserves the center (foliation with closed paths) defined by $\gamma(t_0)$ if and 
only if $\Delta(t,\epsilon)\equiv 0$, for every $t,\epsilon\in\mathbb{C}$.

\medskip

Now, assume that we are in the hyper-elliptic case in the real plane, that is $F(x,y)=y^2+f(x)$, $f\in\mathbb{R}[x]$. Consider a transversal section {$T$} in the axis $y=0$  and $x_0,t_0=f(x_0)$ such 
that $f'(x_0)\neq 0$ and the curve $\gamma(t_0)$ is closed. As for $\epsilon=0$ the foliation defined
by $F(x,y)=t$ consists of closed curves, there exists a first time $t_1(t_0)>0$ such that $\gamma(t_1)$ belongs
to the axis $y=0$. We may extend this function to $t,\epsilon$ by the implicit function theorem. Moreover, 
as $\gamma(-t_1)$ also belong to the axis $y=0$, we also define a function $t_2(t,\epsilon)$ as the first negative time such that $\gamma(t_2)$ intersects the axis $y=0$. Then, the displacement map can be written as the zero dimensional
integral
\[
\begin{split}
\Delta_\epsilon(t)&=\int_{\gamma_\epsilon(t)} dF=\int_{(x(t_1,\epsilon),0)}^{(x(t_2,\epsilon),0)}dF
=F(x(t_2,\epsilon),0)-F(x(t_1,\epsilon),0)\\&=f(x(t_2,\epsilon))-f(x(t_1,\epsilon)).
\end{split}
\]
The tangential version of this problem has been solved in~\cite{PM} for a vanishing cycle with a 
Morse point and in~\cite{GP} for a vanishing cycle in general. Note that in this problem the 
deformation is not polynomial, so we can not directly apply our results.


\end{exam}

\begin{exam}{\bf Moment problem.}
Around 2000, in a series of papers~\cite{BFY98,BFY99,BFY00,BFY00-2}, Briskin, Françoise and Yomdim proposed the problem 
of determing the trigonometric polynomials $a,b$ such that the 
following family of Abel differential equations has a center at the
origin for every $\epsilon\in\mathbb{R}$
\begin{equation}\label{eq:Abel}
x'=a(t) x^2+\epsilon b(t) x^3,
\end{equation}
where an equation of the family is said to have a center at the origin if 
every bounded solution is closed. 

A necessary condition, called \emph{composition condition} \cite{ALl}, is the existence of functions, $\tilde A,\tilde B, h$,
with $h(0)=h(2\pi)$, such that 
\[A(t):=\int_0^t a(s)\,ds=\tilde A(h(t)),\quad B(t):=\int_0^t b(s)\, ds=\tilde B(h(t)).\]

Let us denote $x(t,x_0,\epsilon)$ the solution of \eqref{eq:Abel} determined by 
the initial condition $x(0,x_0,\epsilon)=0$. Assume that for $\epsilon=0$, \eqref{eq:Abel}
has a center at the origin. That is $\int_0^{2\pi} a(t)\,dt=0$. 
Differentiating $x^{-1}(t,x_0,\epsilon)$
with respect to $t$, evaluating at $t=2\pi$ and denoting $t=1/x(0)$, we obtain 
that \eqref{eq:Abel} is a center for every $\epsilon$
if and only if
\begin{equation}\label{eq:AbelInf}
\int_0^{2\pi} \frac{b(t)}{t-\int_0^t a(s)+\epsilon b(s)x(s)\,ds}\,dt\equiv 0,\quad \text{ for every }\epsilon,t\in\mathbb{R}.
\end{equation}
By the change of variables $t\mapsto z=\exp(it)$, and complexifying the variables,
\[
\oint_{|z|=1} \frac{b(z)}{t-\int_0^z a(w)+\epsilon b(w)x(w)\,dw}\,dz\equiv 0,\quad \text{ for every }\epsilon,t\in\mathbb{C}.
\]
By Proposition~3.1 of \cite{ABC}, this is equivalent to 
\[
\sum n_i b(z_i(t,\epsilon))\equiv 0,\quad \text{ for every }\epsilon\in\mathbb{R},t\in\mathbb{C},
\]
where $z_i$ are the preimages of $t$ by $\int_0^z a(w)+\epsilon b(w)x(w)\,dw$,
and $n_i$ are related to the branches of $A^{-1}$. Obviously,
the functions involved are not polynomials, not even rational, but we expect some ideas developped here
could be used for that problem. 
\end{exam}

\begin{exam}{\bf Polynomial moment problem.}
The same problem can be considered when $a,b$ are polynomials 
and a closed solution is a solution $x(t)$ of \eqref{eq:Abel}
such that $x(0)=x(2\pi)$. This polynomial version of the moment problem has been recently 
solved by Pakovich~\cite{PSCC}, proving that if \eqref{eq:Abel} has a center for
every $\epsilon\in\mathbb{R}$, then $a,b$ satisfy the composition condition. 

Pakovich's solution of the polynomial moment problem is based strongly on the solution of the tangential version of the problem~\cite{PM}, 
which ask for centers {\it at first order in $\epsilon$}. More precisely, taking $\epsilon=0$ in \eqref{eq:AbelInf},
a necessary condition for \eqref{eq:Abel} to have a center for every $\epsilon$
is that 
\begin{equation}\label{eq:tc}
\int_0^{1} \frac{b(t)}{t-\int_0^t a(s)\,ds}\,dt\equiv 0,\quad \text{ for every }t\in\mathbb{R}.
\end{equation}
When \eqref{eq:tc} holds, we say that \eqref{eq:Abel} has a tangential center.

\medskip

A stronger condition is to consider $a(t)=a_0(t)+\epsilon b(t)$, and assume that \eqref{eq:tc} holds for 
every $\epsilon\in\mathbb{R}$. It was proposed by Cima, Gasull and Mañosas~\cite{CGM}. They call it \emph{highly persistent center}. They proved that if \eqref{eq:Abel} has a highly persistent 
center, then \eqref{eq:Abel} has a composition center. 

\medskip

Now, we show an alternative proof of this result, using Theorem~\ref{theo:main}.
Denote $f(z)=\int_0^z a(w)\, dw$, $g(z)=\int_0^z b(w)\, dw$, and $t=1/x(0)$. 
By Proposition~8.2 of \cite{ABM1}, 
\[
\int_0^{1} \frac{b(z)}{t-f(z)}\,dz=\int_{C(t)} g(z),
\]
where $C$ is the cycle
\[
C(t)=n_1\sum_{i=1}^{n_0} z_{a_i}(t)-n_0\sum_{i=1}^{n_1} z_{b_i}(t),
\]
where $z_{a_i}(t)$ are all the solutions of $f(z_{a_i}(t))=t$ with $z_{a_i}(t)$ close to $0$ for $t$ close to $f(0)$,
and analogously for $z_{b_i}(t)$ and $1$. 

Now, if $a(z)=a_0(z)+\epsilon b(z)$, then $f(z)=f_0(z)+\epsilon g(z)$, the cycle $C$ becomes $C_\epsilon$. 
So, if \eqref{eq:Abel} has a highly persistent center, then
\[
\int_{C_\epsilon(t)} g(z)\equiv 0,\quad \text{for every }t,\epsilon\in\mathbb{C}.
\]
By Theorem~\ref{theo:main}, there exists $F\in\mathbb{C}[\epsilon,z]$, affine in $\epsilon$, and $h\in \mathbb{C}[z]$, such that $f(z)+\epsilon g(z)=F(h(z),\epsilon)$ and $h(C_\epsilon(t))\equiv 0$. This implies 
that $h(0)=h(1)$, and that there exists $\tilde f,\tilde g$ such that $f=\tilde f\circ h$, $g=\tilde g\circ h$, so the composition condition holds.

\end{exam}

\end{document}